\documentclass[11pt]{article}
\usepackage[utf8]{inputenc}
\usepackage{amsmath, amssymb, amsfonts, amsthm}
\usepackage[none]{hyphenat}
\usepackage{float}
\usepackage{mathtools}
\usepackage[colorlinks=true]{hyperref}
\hypersetup{
  colorlinks=true,
  linkcolor=blue,
  filecolor=magenta,
  urlcolor=cyan,
  citecolor=blue
}
\usepackage{cleveref}
\usepackage[top=25mm, bottom=25mm, left=25mm, right=25mm]{geometry}
\usepackage{tikz}
\usetikzlibrary{decorations.pathreplacing,arrows.meta,positioning,calc}
\usepackage{xcolor}

\definecolor{oiBlue}{HTML}{56B4E9}
\definecolor{oiOrange}{HTML}{E69F00}
\definecolor{oiGreen}{HTML}{009E73}
\definecolor{oiYellow}{HTML}{F0E442}
\definecolor{oiNavy}{HTML}{0072B2}
\definecolor{oiRed}{HTML}{D55E00}
\definecolor{oiPurple}{HTML}{CC79A7}
\definecolor{oiGrey}{HTML}{999999}
\usepackage{comment}

\theoremstyle{plain}
\newtheorem{theorem}{Theorem}        
\newtheorem{lemma}{Lemma}            
\newtheorem{corollary}{Corollary}    
\newtheorem{proposition}{Proposition}

\newtheorem{remark}{Remark}
\newtheorem{problem}{Problem}

\newcommand{\F}{\mathbb{F}}
\newcommand{\PG}{\mathrm{PG}}

\newtheorem{innerthm}{Theorem}
\newenvironment{reptheorem}[1]{%
  \renewcommand\theinnerthm{\ref*{#1}}%
  \innerthm
}{\endinnerthm}

\newtheorem{innerprop}{Proposition}
\newenvironment{repproposition}[1]{%
  \renewcommand\theinnerprop{\ref*{#1}}%
  \innerprop
}{\endinnerprop}

\title{The chromatic number of finite projective spaces}
\author{
Anurag Bishnoi\thanks{Delft University of Technology, Netherlands. \emph{E-mail}: \text{a.bishnoi@tudelft.nl}.} 
\and
Wouter Cames van Batenburg \thanks{Universit\'e libre de Bruxelles, Belgium. Supported by the Belgian National Fund for Scientific Research (FNRS). \emph{E-mail}: \text{w.p.s.camesvanbatenburg@gmail.com}. }
\and 
Ananthakrishnan Ravi\thanks{Delft University of Technology, Netherlands. Supported by an NWO open competition
grant (OCENW.M.22.090). \emph{E-mail}: \text{a.ravi@tudelft.nl}.}
}
\date{}

\begin{document}
\maketitle

\begin{abstract}
The chromatic number of the finite projective space $\mathrm{PG}(n-1,q)$, denoted $\chi_q(n)$, is the minimum number of colors needed to color its points so that no line is monochromatic.
We prove subadditivity of $\chi_q(n)$ with respect to $n$, and then establish the following  stronger recursive bound:
\[
\chi_q(n)\le \chi_q(d)+\chi_q(n+1-d)-1
\] 
for all $1 \leq d < n$, and use it to prove new upper bounds on $\chi_q(n)$.
For $q = 2$, using this recursion we prove that 
\[
\chi_2(n) \le \lfloor 2n/3 \rfloor + 1
\]
for all $n \ge 2$, and we show that this bound is tight for all $n \le 7$. 
In particular, our result recovers all previously known cases for $n \le 6$ and resolves the first open case $n = 7$. 
It also disproves a conjecture of Haddad that $\chi_2(n) = n - 1$ for all $n \geq 4$, in a strong sense. 
On the lower-bound side, using a connection with multicolor Ramsey numbers for triangles, we note that
\[
\chi_2(n) \ge (1 - o(1))\,\frac{n}{\log n}.
\]

We also consider $\chi_q(t;n)$, the minimum number of colors needed to color the points of $\mathrm{PG}(n-1,q)$ with no monochromatic $(t - 1)$-dimensional subspace, and establish an equivalence between $\chi_q(t;n)$ and the multicolor vector-space Ramsey numbers $R_q(t;k)$. 
Using this equivalence together with new upper bounds on $\chi_q(t;n)$, we improve, for every fixed $t$ and $q$, the best known lower bounds on $R_q(t;k)$ from $\Omega_{q,t}(\log k)$ to $\Omega(k)$. 
\end{abstract}

\section{Introduction}\label{sec:intro}

Let $\F_2^n$ denote the $n$-dimensional vector space over $\F_2$ and let $\PG(n-1,2)$ denote the corresponding $(n-1)$-dimensional projective space.
We identify the points of $\PG(n-1,2)$ with the non-zero vectors of $\F_2^n$.
A line in $\mathrm{PG}(n - 1, 2)$ then corresponds to a triple $\{x,y,x+y\}$ with distinct $x,y\in \F_2^n\setminus\{0\}$. 
These lines define a $3$-uniform hypergraph $(V_n, E_n)$ with
\[
  V_n=\F_2^n \setminus \{0\}
  \quad\text{and}\quad
  E_n=\bigl\{\{x,y,x+y\}: x,y\in V_n,\ x\neq y\bigr\},
\]
whose chromatic number is called the chromatic number $\chi_2(n)$ of the binary projective space. 
This is the least number of colors needed to color the points of $\PG(n - 1, 2)$ such that no line is monochromatic. 
Note that a color class contains no three collinear points; that is, each color class is a (projective) \textit{cap} (see \cite{HirschfeldThas2025} for a recent survey on caps).
Therefore, $\chi_2(n)$ is also the minimum number of caps that partition $\PG(n-1,2)$. 
Equivalently, $\chi_2(n)$ is the smallest number of sum-free sets in the abelian group $\mathbb{F}_2^n$ needed to partition the set of non-zero elements.

It can be easily shown that $\chi_2(2) = 2$. 
It is a standard exercise in the context of Property~B and minimal non-2-colorable hypergraphs (see, for example, \cite[Section 1.3]{ZhaoPM}) that $\chi_2(3) = 3$, that is, the chromatic number of the Fano plane is $3$. 
The next case is $\chi_2(4) = 3$, which follows from the well-known partition of $\mathrm{PG}(3,2)$ into elliptic quadrics \cite{Ebert1985}.

The chromatic number of binary projective spaces was studied by Rosa~\cite{Rosa1970-1, Rosa1970-2} in his investigation of more general Steiner triple systems (see \cite{BrandesPhelpsRodl1982, BruenHaddadWehlau1998} for further results in this direction).
He showed $\chi_2(5)=4$, and asked if $\chi_2(6)=5$; this was confirmed by Fug\`ere, Haddad and Wehlau~\cite{Fugere1994}.
Combining our upper bound with a Ramsey lower-bound argument in Section~\ref{sec:lower}, we obtain $\chi_2(7)=5$, thus solving the next open case.

\textit{In this paper}, we prove new bounds on $\chi_2(n)$, recovering all the small cases proved earlier and improving the previous general upper bound. 
By coloring the complement of a hyperplane with one color and recursively iterating inside the hyperplane, it can be easily shown that $\chi_2(n) \leq n$. For example, let $i=1,\dots,n$, define 
\[
C_i=\{x=(x_1,\dots,x_n)\in \F_2^n\setminus\{0\} : x_j=0\ \text{for all } j<i,\ \text{and } x_i=1\}.
\]
Then $\PG(n - 1, 2) = \bigsqcup_{i=1}^n C_i$, and each $C_i$ is a cap. 
This bound can be improved to \[\chi_2(n) \le n-1\ \text{for all}\ n\ge 4,\] using the recursion $\chi_2(n)\leq \chi_2(n - 1) + 1$ that follows from the above argument, together with the well-known base case of $n = 4$. 

Since $\chi_2(4)$, $\chi_2(5)$, and $\chi_2(6)$ are $3$, $4$, and $5$, respectively, Haddad conjectured in \cite{Haddad1999} that $\chi_2(n)$ must increase by $1$ when the dimension increases by $1$ for all $n \geq 4$. 
Rosa \cite{Rosa2019} mentions that this conjecture was disproved by Blokhuis.
We were unable to find the argument of Blokhuis, but the recursion $\chi_2(n) \leq \chi_2(d) + \chi_2(n - d)$ that follows from taking quotient spaces (we prove this in a more general setting in Lemma~\ref{lem:qtrecursion2}), combined with $\chi_2(4) = 3$, implies that 
$\chi_2(n) \le 3n/4 + O(1)$ (see \Cref{cor:chi2-3n4}). 
We prove the following improvement to this upper bound. 
\begin{theorem}\label{thm:main}
For all $n \geq 2$
\[\chi_2(n)\le \lfloor 2n/3 \rfloor + 1 .\]
\end{theorem}
We will see in Section \ref{sec:lower} that this upper bound is tight for all $n \leq 7$. 
The key ingredient in the proof of \Cref{thm:main} is the following improved recursion, which we state more generally for $\chi_q(n)$, the minimum number of colors needed to color the points of \(\mathrm{PG}(n-1,q)\) so that no projective line is monochromatic. 
For every prime power $q$ and all $1\le d\le n-1$,
\begin{equation}\label{eq:recursion}
\chi_q(n)\le \chi_q(d) + \chi_q(n+1-d)-1.
\end{equation}

Starting from the base case of $\chi_2(4)=3$, this yields \Cref{thm:main}.
See Section \ref{sec:recursion} for the proof of the recursion, which uses a ``reserved color" and a scalar ``lift". 
This mechanism is in the same spirit as the ``templates"-based approach for Schur numbers \cite{Rowley2021,Ageron2021}.

For lower bounds on $\chi_2(n)$, we use a connection with multicolor Ramsey numbers for triangles. 
Recall that $R(3;k)$ denotes the smallest $N$ such that any $k$-coloring of the edges of a complete graph $K_N$ yields a monochromatic triangle. 
Around a decade before Ramsey's foundational work, Schur~\cite{Schur1917} implicitly showed that 
$\Omega(2^k)\le R(3;k)\le O(k!)$. 
Erd\H{o}s later conjectured that $R(3;k)=2^{\Theta(k)}$.
We refer to~\cite{NR01} for a historical account.

If we could properly color $\PG(n-1,2)$ with $k$ colors, then by identifying the $2^n$ vectors of $\F_2^n$ with the vertices of a $K_{2^n}$, we could color each edge $\{u,v\}$ of $K_{2^n}$ by the color of the difference $(v-u) \in \F_2^n \setminus \{0\}$, identified as a point of $\PG(n -1, 2)$.
So, a proper $k$-coloring of $\PG(n-1,2)$ would produce a $k$-edge-coloring of $K_{2^n}$ with no monochromatic triangles, and thus
\begin{equation}\label{eq:Ramsey_connection}
\chi_2(n)\leq k \implies R(3;k) > 2^n.
\end{equation}
Our upper bound in Theorem~\ref{thm:main} then implies that
\begin{equation}\label{eq:our-r3k-lowerbound}
R(3;k) > 2^{3(k-1)/2} > 2.828^{k-1}.
\end{equation}
This connection between partitioning an abelian group into sum-free sets and $R(3;k)$ goes back to Abbott and Hanson~\cite{AbbottHanson1972} and ultimately Schur~\cite{Schur1917}.

In the other direction, using \cref{eq:Ramsey_connection} and the classical asymptotic upper bound \cite{GreenwoodGleason1955} of $R(3;k) \leq e k! + 1$, we get the following lower bound on $\chi_2(n)$. All logarithms in this paper have base $2$.
\begin{proposition}\label{prop:chin-lower}
\[
\chi_2(n) \ge (1 - o(1))\,\frac{n}{\log n}.
\]

\end{proposition}
This highlights a wide gap: current lower bounds for $R(3;k)$ are exponential in $k$ \cite{Rowley2021,Ageron2021}, while the best upper bounds are super-exponential. 
Any sublinear improvement for $\chi_2(n)$ would push these lower bounds of $R(3;k)$ beyond the exponential regime, which would disprove the conjecture of Erd\H{o}s. 
At a concrete level, the only exact values known are $R(3;1)=3$, $R(3;2)=6$, and $R(3;3)=17$. For $k=4$, the current status is $51 \ \le\ R(3;4) \ \le\ 62$ with the lower bound due to Chung~\cite{Chung1973}, the upper bound due to Fettes, Kramer and Radziszowski~\cite{FettesKramerRadziszowski2004}, improving the earlier $R(3;4)\le 64$ of S\'anchez-Flores~\cite{SanchezFlores1995}; it is conjectured in~\cite{XuRadziszowski2016} that $R(3;4)=51$ (see also the dynamic survey~\cite{RadziszowskiSurvey2017}).
We will see below that these small cases can be used to find good bounds on $\chi_2(n)$.

Since the seminal upper bound of Greenwood–Gleason \cite{GreenwoodGleason1955}, progress in the general upper bounds of $R(3;k)$ has only focused on improving the leading constant in front of $k!$.
The connection to $R(3;k)$ highlights that the chromatic number of (binary) projective spaces is a difficult quantity to pin down, exactly or even asymptotically.
Nevertheless, we hope that the new upper bound of \Cref{thm:main} opens the door to further improvements in the asymptotics of $R(3; k)$. 

\textit{Beyond the binary case}, we also study the chromatic number of projective spaces over general finite fields and allow higher-dimensional forbidden subspaces.
For a prime power $q$ and an integer $t\ge 2$ we write $\chi_q(t;n)$ for the chromatic number of the hypergraph on $\PG(n-1,q)$ whose edges are the sets of points of $(t-1)$-dimensional projective subspaces.
For $t=2$ this is the chromatic number $\chi_q(n)$ of $\PG(n-1,q)$ with respect to lines.
In the general setting, we prove the following recursion in Lemma~\ref{lem:qtrecursion2}.
\[
    \chi_q(t; n) \leq \chi_q(t; n - d) + \chi_q(t; d).
\]
For $t\ge3$, we combine this recursion with a two-coloring of $\PG(2t-1,q)$ coming from elliptic quadrics to get the following bounds.
\begin{theorem}\label{thm:chiq-upper-bound_t}
For every prime power $q$ and integer $t\ge3$,
\[
    \chi_q(t;n) \le \frac{n}{t}+O_t(1).
\]
\end{theorem}

For $t=2$, we use an improved recursion \eqref{eq:recursion} and some known base cases to obtain the following bounds.
\begin{theorem}\label{thm:chiq-upper-bound}
For every fixed prime power $q\ge 7$, 
\[
\chi_q(n)\le \frac{2}{q}n+O_q(1),
\]
for $q=5$, 
\[
\chi_5(n)\le n/3+O(1),
\]
and, for $q=3$ and $q=4$, 
\[
\chi_q(n)\le n/2+O(1).
\]
\end{theorem}

We then show an equivalence between $\chi_q(t;n)$ and multicolor vector space Ramsey numbers $R_q(t;k)$.
These Ramsey numbers were originally introduced in \cite{GrahamLeebRothschild1972} (also see \cite{Spencer1979}) and have recently been studied in \cite{Frederickson2023, Hunter2025}.
Recall that $R_q(t;k)$ is the smallest integer $n$ such that in every $k$-coloring of the $1$-dimensional linear subspaces of $\mathbb{F}_q^n$, there is a monochromatic $t$-dimensional linear subspace.
Equivalently, $R_q(t;k)$ is the smallest integer $n$ such that in every $k$-coloring of the points of $\PG(n-1,q)$ there exists a monochromatic $(t-1)$-dimensional projective subspace.
Therefore, 
\[\chi_q(t; n) \leq k \iff R_q(t; k) > n.\]    

Using our results, we  thus obtain new lower bounds for $R_q(t;k)$.
\begin{theorem}\label{thm:Rqt-lower}
For every prime power $q$, every integer $t\ge3$, and every integer $k\ge1$,
\[
R_q(t;k)>tk-O_{t}(1).
\]
\end{theorem}

Note that for $t=2$, the lower bound $R_q(2;k)>k$ follows from $\chi_q(n)\le n$, and thus the improved upper bounds in Theorem~\ref{thm:chiq-upper-bound} give better lower bounds on $R_q(2;k)$. 
In particular, we obtain
\[
R_2(2;k)>\frac32 k-O(1),\quad
R_q(2;k)>2k-O_q(1)\ (q=3,4),\quad
R_5(2;k)>3k-O(1),
\]
and
\[
R_q(2;k)>\frac q2 k-O_q(1)
\quad\text{for every fixed prime power }q\ge7.
\]
Together with the bound $R_q(t;k)>tk-O_t(1)$ for $t\ge3$, this improves the lower bounds on $R_q(t;k)$ from $\Omega_q(\log k)$ to $\Omega(k)$ for every fixed $q$ and $t$.

We also record an upper bound in the $q=t=2$ case, obtained from the connection with classical multicolor Ramsey numbers for triangles.
\begin{theorem}\label{thm:R22k-upper}
For all integers $k\ge 1$
\[
 R_2(2;k)=O(k\log k).
\]
\end{theorem}

This improves the upper bound $R_2(2;k)=O(k\log^7 k)$ that follows from the more general bound given by Hunter and Pohoata \cite[Theorem 1.4]{Hunter2025}, which was obtained using the breakthrough result of Kelley and Meka \cite{KelleyMeka2023}.
Note that from the discussion around \cite[Theorem 4.1]{Hunter2025} one can also deduce the upper bound of $O(k \log k)$ on $R_2(2; k)$, but it is not explicitly stated.

Our paper is organized as follows.
In Section~\ref{sec:nonbinary} we introduce the general chromatic numbers $\chi_q(t;n)$, establish the basic recursion, and prove the upper bound for $t\ge3$.
In Section~\ref{sec:recursion} we prove the improved recursion for $t=2$ and derive the upper bounds on $\chi_q(n)$.
In Section~\ref{sec:lower} we develop the connection to classical multicolor Ramsey numbers for triangles and determine exact values of $\chi_2(n)$ for small $n$.
In Section~\ref{sec:vsrn} we establish the link to multicolor vector space Ramsey numbers $R_q(t;k)$ and use our bounds on $\chi_q(t;n)$ to improve the lower bounds for $R_q(t;k)$.
In Section~\ref{sec:conclusion} we discuss future work and related open problems.

\section{General upper bounds on the chromatic number}\label{sec:nonbinary}

We use the convention that all dimensions are projective.
Recall that we obtain a projective space $\PG(n-1,\F)$ from the $n$-dimensional vector space $V$ over the field $\F$ in the following way:
The points, lines, $\ldots$, $(d-1)$-spaces of $\PG(n - 1, \mathbb{F})$ are the $1$-dimensional, $2$-dimensional, $\dots$, $d$-dimensional vector subspaces of $V$.
For $v = (v_1,\dots,v_n) \in \F_q^n \setminus \{0\}$, we write
\[
  [v] = [v_1 : \dots : v_n] \in \PG(n-1,q)
\]
for the point corresponding to the $1$-dimensional subspace $\langle v \rangle \subset \F_q^n$.
Thus, $[v] = [\lambda v]$ for every $\lambda \in \F_q^\times$.

Fix a prime powers $q$ and an integer $t\ge 2$.
For each integer $n\ge 1$, let $\chi_q(t;n)$ denote the smallest $k$ such that the points of $\PG(n-1,q)$ can be colored with $k$ colors in such a way that no $(t-1)$-dimensional projective subspace is monochromatic.
Equivalently, this is the chromatic number of the hypergraph whose vertices are the points of $\PG(n-1,q)$ and whose hyperedges are the sets of points of each $(t-1)$-dimensional projective subspace.

For $q=2$ and $t=2$, a color class is a cap (no three collinear), so a proper coloring is the same as a partition into caps. 
For $q\ge 3$ and $t=2$, a color class must avoid a full $(q+1)$-point line.
In other words, its complement is a blocking set \cite{Beutelspacher1980}. 
We denote $\chi_q(2;n)$ as $\chi_q(n)$.
In particular, partitions into caps give upper bounds on $\chi_q(n)$, but the true value of $\chi_q(n)$ can be smaller.

We can prove the following upper bounds.
Let $S$ be a projective subspace of $\PG(n-1,q)$ of dimension $n-t$.
Then $S\cong\PG(n-t,q)$, so we can color $S$ properly
with $\chi_q(t;n-t+1)$ colors.
Color all points outside $S$ with one new color. Then no $(t-1)$-dimensional projective subspace is monochromatic.
Hence, for all prime powers $q\ge 2$ and for $n \ge t\ge 2$,
\begin{equation}\label{qtrecursion1}
  \chi_q(t;n)\ \le\ \chi_{q}(t;n-t+1)+1.  
\end{equation}

In fact, we can do better. 
We have the following recursion for $\chi_q(t;n)$ that gives an improved upper bound in general. 

\begin{lemma}\label{lem:qtrecursion2}
For all prime powers $q\ge 2$, for all integers $t\ge2$, and for all integers $1\le d\le n-1$,
\begin{equation}
    \chi_q(t; n) \leq \chi_q(t; n - d) + \chi_q(t; d).
\end{equation}
\end{lemma}
\begin{proof}
Represent the points of $\PG(n-1,q)$ as $[v_1:\dots:v_n]$.
Define
\[
A \;=\; \bigl\{[v_1:\dots:v_n] : v_{n-d+1}=\dots=v_n=0\bigr\},
\]
\[
B \;=\; \bigl\{[v_1:\dots:v_n] : v_1=\dots=v_{n-d}=0\bigr\}.
\]
Then $A\cong\PG(n-d-1,q)$ and $B\cong\PG(d-1,q)$.
If $v=[v_1:\dots:v_n]  \in \PG(n-1,q)$ is a point, then we write
\[
v=[a:b].
\]
The point $[a:b]$ lies in $A$ if and only if $b=0$.

Let $c_A$ be a proper coloring of $A$ with $\chi_{q}(t;n-d)$ colors, and let $c_B$ be a proper coloring of $B$ with $\chi_q(t;d)$ colors.
Use disjoint color palettes for $c_A$ and $c_B$.
Define a coloring $c$ of $\PG(n-1,q)$ by
\[
c([a:b]) \;=\;
\begin{cases}
c_A([a:0]), & \text{if } b=0,\\[3pt]
c_B([b]),   & \text{if } b\ne 0.
\end{cases}
\]
Thus, $c$ uses $\chi_{q}(t;n-d)+\chi_q(t;d)$ colors.

We show that no $(t-1)$-dimensional projective subspace is monochromatic.
Let $S\subseteq\PG(n-1,q)$ be a $(t-1)$-dimensional projective subspace.
We distinguish three cases: $S\cap A=\emptyset$, $S\subseteq A$, and neither of these holds.

\begin{itemize}
    \item  Suppose that $S\cap A=\emptyset$.
    Then every point of $S$ has the form $[a:b]$ with $b\ne 0$, and hence $c([a:b])=c_B([b])$.
    Consider the projection $\varphi([a:b]) = b$, which maps points of $\PG(n-1,q)$ to points of $\PG(d-1,q)$.
    We show that $\varphi\big{|}_S$ is injective, which would imply that $\varphi(S)$ is a monochromatic $(t-1)$-space in $B$; a contradiction.
    
    Say, $\varphi\big{|}_S$ is not injective, i.e., there exists $[a_1:b_1] \neq [a_2:b_2]$ on $S$ such that $\varphi([a_1:b_1]) = \varphi([a_2:b_2]) $, and hence $b_1 = \lambda b_2$, for some $\lambda \neq 0$.
    Then, the point with coordinates $[a_1:b_1] - \lambda [a_2:b_2] = [a_1 - \lambda a_2:0]$ lies in $S \cap A$, which is a contradiction.
    
    \item  If $S$ contains a point $x=[a_1:b_1]$ with $b_1=0$ and a point $y=[a_2:b_2]$ with $b_2\ne 0$, then $c(x)$ is a color from the palette of $c_A$ and $c(y)$ is a color from the palette of $c_B$.
    Since the palettes are disjoint, $c(x)\neq c(y)$.
    
    \item  If $S\subseteq A$, then the coloring $c$ on $S$ coincides with $c_A$.
    Since $c_A$ is proper for $(t-1)$-spaces in $A$, $S$ is not monochromatic.
    
\end{itemize}

In all cases $S$ is not monochromatic.
Therefore, $c$ is a proper coloring of $\PG(n-1,q)$, and this gives
\[
\chi_q(t;n)\ \le\ \chi_q(t;n-d)+\chi_q(t;d).
\]
\end{proof}

\begin{remark}
    We can give a geometric coordinate-free explanation of the recursion as follows.
    Take an $(n - d - 1)$-dimensional projective subspace $A$ of $\PG(n - 1, q)$, and let $B = PG(n - 1, q)/A \cong PG(d - 1, q)$ be the quotient space with respect to $A$. 
    Color the points of $A$ according to a minimal proper coloring of $A$ and the points outside $A$ according to a minimal proper coloring of $B$, where a point $x$ outside $A$ in $\PG(n - 1, q)$ is identified with the point $\langle x, A \rangle$ of $B$. 
    This gives a proper coloring of the points of $\mathrm{PG}(n - 1, q)$ using $\chi_q(t; n - d) + \chi_q(t; d)$ colors. 
\end{remark}

\begin{remark}
    This upper bound shows that $\chi_q(t; n)$ is subadditive in $n$, and thus $\lim_{n \rightarrow \infty} \chi_q(t; n)/n$ exists. 
\end{remark}

\begin{corollary}\label{cor:chi2-3n4}
For all integers $n\ge 2$,
\[
  \chi_2(n) \le \lceil 3n/4 \rceil.
\]
\end{corollary}
\begin{proof}
Applying \Cref{lem:qtrecursion2} with $q=t=2$, $d=4$ and $\chi_2(4)=3$, we have
\[
\chi_2(n)\le 3+\chi_2(n-4) \qquad (n\ge 5).
\]
Write $n=4m+r$ with $r\in\{1,2,3,4\}$.
Iterating the inequality $m$ times gives
\[
\chi_2(n)\le 3m + \chi_2(r).
\]
Using $\chi_2(r) = r$ for all $r \in \{1, 2, 3\}$ and $\chi_2(4)=3$, we get the stated bound.
\end{proof}
We now prove new upper bounds on $\chi_q(t;n)$ for $t\ge3$ and all $q$. 
\begin{reptheorem}{thm:chiq-upper-bound_t}
For every prime power $q$ and every integer $t\ge3$,
\[
    \chi_q(t;n) \le \frac{n}{t}+O_t(1).
\]
\end{reptheorem}
\begin{proof}
Let $Q$ be an elliptic quadric in $\PG(2t-1,q)$.
It is well-known that $Q$ contains no $(t-1)$-dimensional projective subspace (see for example \cite[Section 4.2]{Ball2015}).

We now show that the complement of $Q$ contains no projective plane.
Let $\pi$ be a projective plane in $\PG(2t-1,q)$.
Then the points of $\pi$ are given by solving $2t-3$ homogeneous equations of degree $1$. 
Therefore, by Chevalley--Warning, the degree $2$ equation corresponding to $Q$ and these $2t - 3$ linear equations, must have a non-zero common solution, since $2 + 2t - 3 = 2t - 1 < 2t$. Therefore, $\pi$ contains a point of $Q$ and thus every projective plane meets $Q$ non-trivially. 
Since $t \geq 3$, this implies that the complement of $Q$ does not contain a $(t - 1)$-dimensional subspace.

Now color the points of $Q$ with one color and the remaining points with a
second color.
By the reasoning above, neither color class contains a
$(t-1)$-dimensional projective subspace.
Thus
\[
    \chi_q(t;2t)\le2.
\]
Applying Lemma~\ref{lem:qtrecursion2} iteratively with this base coloring gives
\[
    \chi_q(t;n)\le \frac{n}{t}+O_t(1).
\]
\end{proof}

\section{Improved recursion for lines}\label{sec:recursion}
In the previous section we proved the recursion
\[
    \chi_q(t;n) \le \chi_q(t;n-d)+\chi_q(t;d),
\]
which works for all $t$.
In this section we improve it in the case $t=2$.
Let $\chi_q(n) \coloneqq \chi_q(2;n)$.

\begin{lemma}\label{lem:recursion}

For every prime power $q$ and every integer $1\le d\le n-1$,
\[
\chi_q(n)\le \chi_q(d)+ \chi_q(n+1-d)-1.
\]
\end{lemma}

\begin{proof}
Let $n-1\ge d\ge 1$.
Represent the points of $\PG(n-1,q)$ as $[v_1:\dots:v_n]$.
Define
\[
A \;=\; \bigl\{[v_1:\dots:v_n] : v_{n-d+1}=\dots=v_n=0\bigr\},
\]
\[
U \;=\; \bigl\{[v_1:\dots:v_n] : v_1=\dots=v_{n-d}=0\bigr\}.
\]
Then $A\cong\PG(n-d-1,q)$ and $U\cong\PG(d-1,q)$.
Every nonzero vector $v\in \mathbb F_q^n$ can be written uniquely as $v=(a,u)$, where $a\in \mathbb F_q^{n-d}$ and $u\in \mathbb F_q^d$. We write the corresponding projective point as $[a:u]$.
If $a$ is nonzero, then $[a:0]\in A$, and if $u$ is nonzero, then $[0:u]\in U$.

Take a proper $\chi_q(d)$-coloring $c_U$ of $U\cong\PG(d-1,q)$ and fix a reserved color $r$ of $c_U$.
Let $V_U$ be a set of vectors constructed from $U$ by fixing a particular coordinate vector for each point in $U$, say by taking the vector whose first non-zero coordinate is $1$.
Then for each vector $u \in \mathbb{F}_q^{d} \setminus \{0\}$, there exists a unique scalar $\lambda$ such that $u = \lambda v$ for some $v \in V_U$.
Define a function
\[
t:\mathbb{F}_q^{d} \longrightarrow \mathbb{F}_q
\]
by
\[
t(u) \;=\;
\begin{cases}
0, & u=0,\\
0, & u\neq 0 \text{ and } c_U([u])\neq r.\\
\lambda, & u = \lambda v  \text{ for some } v \in V_U \text{ and } c_U([u])=r,
\end{cases}
\]
By construction, for every $\lambda\in\mathbb F_q$ and every $u\in\mathbb F_q^d$, we have $t(\lambda u)=\lambda t(u)$.
Indeed, this is clear if $u=0$ or if $c_U([u])\neq r$. If $c_U([u])=r$ and $u=\mu v$ with $v\in V_U$, then $\lambda u=(\lambda\mu)v$, so $t(\lambda u)=\lambda\mu=\lambda t(u)$.

Now extend the underlying vector space of $A$, $\mathbb F_q^{n-d}$ to $\mathbb F_q^{n-d+1}$, and let $A^+\cong\PG(n-d,q)$ be the corresponding projective space.
Take a proper $\chi_q(n+1-d)$-coloring $c_{A^+}$ of $A^+$, using a palette disjoint from the colors of $c_U$.
Define
\[
c([a:u]) \;=\;
\begin{cases}
c_U([u]), & \text{if } u\neq 0 \text{ and } c_U([u])\neq r,\\[3pt]
c_{A^+}\bigl([a:t(u)]\bigr), & \text{otherwise}
\end{cases}
\]
Thus, $r$ is never used, and the total number of colors is at most $\chi_q(d)-1+\chi_q(n+1-d)$.
Moreover, this function is well-defined since $t(\lambda u) = \lambda t(u)$ which ensures that $[a:t(u)]$ is independent of the choice of representative for $[a:u]$.

We show that $c$ is a proper coloring.
Assume a line $\ell\subset\PG(n-1,q)$ is monochromatic in color $\alpha$.
Let $V_\ell\le \mathbb F_q^n$ be the $2$-dimensional vector subspace corresponding to the projective line $\ell$.
Define the linear projection $\varphi:\mathbb F_q^n=\mathbb F_q^{n-d}\oplus \mathbb F_q^d\longrightarrow \mathbb F_q^d$
by $\varphi(a,u)=u.$

\emph{Case 1: $\alpha$ is a color of $c_U$.}
By the definition of $c$, every point $[a:u]\in \ell$ satisfies $u\neq 0$ and $c_U([u])=\alpha\neq r.$
In particular, no nonzero vector of $V_\ell$ has $U$-component equal to $0$. Equivalently, $\ker(\varphi|_{V_\ell})=\{0\}.$
Thus $\varphi|_{V_\ell}$ is injective and its image is a $2$-dimensional vector subspace, which corresponds to a line of $U$. 
Every point of this line must have color $\alpha$ under $c_U$.
This gives a monochromatic projective line in $U$ for the coloring $c_U$, contradicting the fact that $c_U$ is proper.

\emph{Case 2: $\alpha$ is a color of $c_{A^+}$.}
Then for every $x=[a:u]\in\ell$ we have $u=0$ or $c_U([u])=r$, 
and 
$c(x)=c_{A^+}([a:t(u)])=\alpha$.
Again consider the projection map $\varphi$. 
If its image is a line of $U$, then it must be an $r$-colored line of $U$, contradicting the fact that $c_U$ is a proper coloring. 
If the projection is just the zero vector, then $t(u) = 0$ and we get a monochromatic line in $A^+$, which is again a contradiction. 
Therefore, the projection must be a single point of $U$, that is, a $1$-dimensional vector subspace.
Define $T: V_\ell \rightarrow \mathbb{F}_q^{n - d + 1}$ by $T(a, u) = (a, t(u))$.
We show that this map is a linear injective map and hence its image is also a $2$-dimensional vector space, thus giving us a monochromatic line of $A^+$, a contradiction.

Let $(a_1,u_1)$ and $(a_2,u_2)$ be two vectors in $V_\ell$. 
Since $\varphi(V_\ell)$ is a single $1$-dimensional vector subspace, there exists some $w\in V_U$ such that every $u\in \varphi(V_\ell)$ can be written uniquely as $u=sw$ for some $s\in\mathbb F_q$.
Thus $u_1=s_1w$ and $u_2=s_2w$ for some $s_1,s_2\in\mathbb F_q$.

Since every nonzero $u\in \varphi(V_\ell)$ corresponds to a point colored by the second rule, the point $[w]$ has color $r$ under $c_U$. Therefore, by the definition of $t$, if $u_i=s_iw$, then $t(u_i)=s_i$.
Say $T(a_1,u_1)=T(a_2,u_2)$. Then $a_1=a_2$ and $t(u_1)=t(u_2)$.
Since $t(u_1)=s_1$ and $t(u_2)=s_2$, this implies $s_1=s_2$, and hence $u_1=u_2$. Therefore $(a_1,u_1)=(a_2,u_2)$, showing that $T$ is injective.

We now show linearity.
Let $\lambda_1,\lambda_2\in\mathbb F_q$. Since $u_1=s_1w$ and $u_2=s_2w$, we have $\lambda_1u_1+\lambda_2u_2=(\lambda_1s_1+\lambda_2s_2)w$. Hence, by the definition of $t$, we get $t(\lambda_1u_1+\lambda_2u_2)=\lambda_1s_1+\lambda_2s_2=\lambda_1t(u_1)+\lambda_2t(u_2)$.
Therefore,
\[
T(\lambda_1(a_1,u_1)+\lambda_2(a_2,u_2))
=
(\lambda_1a_1+\lambda_2a_2,t(\lambda_1u_1+\lambda_2u_2))
=
\lambda_1T(a_1,u_1)+\lambda_2T(a_2,u_2).
\]
Thus $T$ is linear.

Since $T$ is linear and injective, $T(V_\ell)$ is a $2$-dimensional vector subspace of $\mathbb F_q^{n-d+1}$. Hence $T(V_\ell)$ corresponds to a projective line in $A^+$.
For every nonzero $(a,u)\in V_\ell$, the point $[a:u]$ lies on $\ell$, and by assumption it has color $\alpha$. Therefore $c_{A^+}([a:t(u)])=\alpha$, so every point of this line has color $\alpha$ under $c_{A^+}$.
This contradicts the properness of $c_{A^+}$.

\vspace{2ex}
\noindent
Both \textit{Case 1} and \textit{Case 2} lead to a contradiction, implying that $c$ must be a proper coloring of $\mathrm{PG}(n - 1, q)$, and thus
\[
\chi_q(n)\;\le\; (\chi_q(d)-1)+\chi_q(n+1-d)\;=\;\chi_q(d)+\chi_q(n+1-d)-1. \qedhere
\]
\end{proof}

\begin{figure}[!t]
  \centering
  \begin{tikzpicture}[font=\small,scale=0.75, every node/.style={scale=0.75}]
      \def\W{4.2}\def\H{9.0}\def\dH{3.3}
      \fill[oiBlue!22]  (0,\dH) rectangle (\W,\H);  
      \fill[oiGreen!22] (0,0)   rectangle (\W,\dH); 
      \draw[thick] (0,0) rectangle (\W,\H);
      \draw[thick] (0,\dH) -- (\W,\dH);
      \node[above] at ({0.5*\W},\H) {$\PG(n-1,q)$};
      \draw[decorate,decoration={brace,amplitude=5pt,mirror}]
        (\W+0.35,\dH) -- (\W+0.35,\H)
        node[midway,xshift=25pt] {$n-d$};
      \draw[decorate,decoration={brace,amplitude=5pt,mirror}]
        (\W+0.35,0) -- (\W+0.35,\dH)
        node[midway,xshift=15pt] {$d$};
      \node[anchor=east] at (-0.3,{\H-0.3}) {$v_1$};
      \node[anchor=east] at (-0.3,{\H-1.0}) {$\vdots$};
      \node[anchor=east] at (-0.3,{\dH+0.3}) {$v_{\,n-d}$};
      \node[anchor=east] at (-0.3,{\dH-0.3}) {$v_{\,n-d+1}$};
      \node[anchor=east] at (-0.3,0.9) {$\vdots$};
      \node[anchor=east] at (-0.3,0.3) {$v_n$};
      \node at ({0.5*\W},{0.5*(\H+\dH)}) {$A \cong \PG(n-d-1,q)$};
      \node at ({0.5*\W},{0.5*\dH}) {$U \cong \PG(d-1,q)$};
  \end{tikzpicture}
  \caption{Coordinate slice of $\PG(n-1,q)$: $A$ (first $n\!-\!d$ coordinates) and $U$ (last $d$)}
  \label{fig:fig1}
\end{figure}
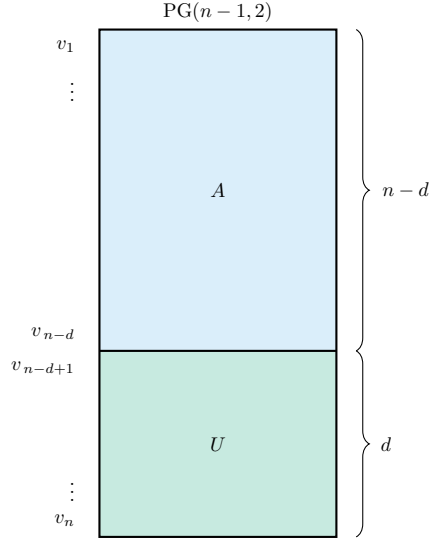

\begin{figure}[!t]
  \centering
  \begin{tikzpicture}[font=\small,scale=0.75, every node/.style={scale=0.75}]
    \def\W{4.2}\def\H{9.0}\def\dH{3.3}
    \def\Gap{1.0}      
    \def\Extra{0.6}    
    \def\RSx{0.7}      

    \pgfmathsetmacro{\HB}{\dH}                    
    \pgfmathsetmacro{\HtopBase}{\H - \dH}
    \pgfmathsetmacro{\HT}{\HtopBase + \Extra}     
    \pgfmathsetmacro{\tY}{\HB+\Gap+\Extra/2}      

    \fill[oiGreen!20] (0,0) rectangle (\W,\HB);
    \fill[oiRed!50]   (\W-\RSx,0) rectangle (\W,\HB);

    \fill[oiYellow!45] (0,{\HB+\Gap}) rectangle (\W,{\HB+\Gap+\Extra});               
    \fill[oiBlue!20]   (0,{\HB+\Gap+\Extra}) rectangle (\W,{\HB+\Gap+\HT});           

    \draw[thick] (0,0) rectangle (\W,\HB);
    \node[anchor=south west] at (0.7,\HB-2) {$U \cong \PG(d-1,q)$};

    \draw[thick] (0,{\HB+\Gap}) rectangle (\W,{\HB+\Gap+\HT});
    \node[anchor=south west] at (0.75,{\HB+\Gap+\HT-2}) {$A^+ \cong \PG(n-d,q)$};

    \node at ({\W-\RSx/2+1.75},{\HB/2}) {reserved color $r$};

    \draw[decorate,decoration={brace,amplitude=6pt,mirror,raise=2pt}]
      ({\W-\RSx},\HB)--(0,\HB) 
      node[midway,yshift=15pt] {uses $\chi_q(d)-1$ colors};

    \draw[decorate,decoration={brace,amplitude=6pt,mirror,raise=2pt}]
      (\W,{\HB+\Gap+\HT})--(0,{\HB+\Gap+\HT})
      node[midway,yshift=18pt] {uses $\chi_q(n+1-d)$ colors};

    \node[anchor=west,align=left] (tcase) at ({\W+1.0},{\tY})
      {$\begin{cases}
         t=0 & \text{if $u=0$},\\
         t=0 & \text{$c_U([u])\neq r$} \\
         t= \lambda & \text{$u=\lambda v$, $v\in V_U$, and $c_U([u])=r$}
       \end{cases}$};

    \draw[-{Stealth[length=2mm]}] ([xshift=-4pt]tcase.west) -- (\W,{\tY});
    
  \end{tikzpicture}
  \caption{Reserved color $r$ in $U$ (vertical strip) and routing to $A^+$ via the extra coordinate $t$.}
  \label{fig:fig2}
\end{figure}
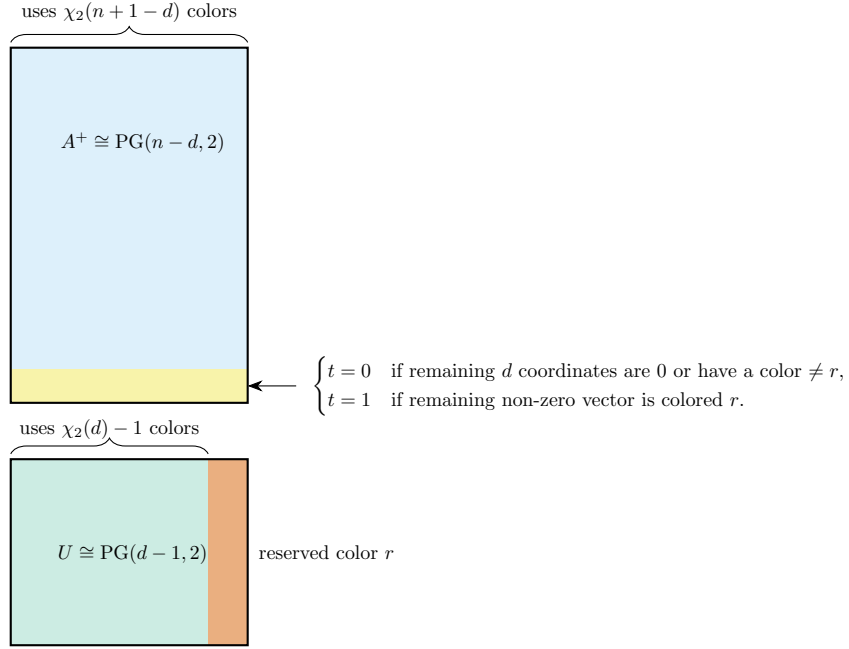

\begin{reptheorem}{thm:main}
For all $n \geq 2$
\[
  \chi_2(n)\le \lfloor 2n/3 \rfloor + 1 .
\]
\end{reptheorem}

\begin{proof}
Applying \Cref{lem:recursion} with $q=2$, $d=4$ and $\chi_2(4)=3$, we have
\[
\chi_2(n)\le 2+\chi_2(n-3) \qquad (n\ge 4).
\]
Write $n=3m+r$ with $r\in\{1,2,3\}$.
Iterating the inequality $m$ times gives
\[
\chi_2(n)\le 2m + \chi_2(r).
\]
Using $\chi_2(r) = r$ for all $r \in \{1, 2, 3\}$, we get the stated bound.
\end{proof}

We now prove the upper bounds for $q\ge 3$.

\begin{reptheorem}{thm:chiq-upper-bound}
For every fixed prime power $q\ge 7$, 
\[
\chi_q(n)\le \frac{2}{q}n+O_q(1).
\]
For $q=5$, 
\[
\chi_5(n)\le n/3+O(1).
\]
For $q=3$ and $q=4$, 
\[
\chi_q(n)\le n/2+O(1).
\]
\end{reptheorem}

\begin{proof}
F\"uhrer and Taranchuk \cite{FuhrerTaranchuk2024} proved that, for every prime power $q$,
\[
\chi_q\!\left(\binom{q}{2}+1\right)\le q.
\]
Applying Lemma~\ref{lem:recursion} with $d=\binom{q}{2}+1$ and iterating gives
\[
\chi_q(n)\le \frac{2}{q}n+O_q(1).
\]
For $q\ge 7$, this gives the first result.

A non-trivial blocking set, which is a set of points that has nonempty intersection with every line of $\PG(n-1,q)$ while not containing any line, exists in $\PG(3,5)$ \cite{Rajola1988}.
Color the points of this blocking set with one color and the remaining points with another color, and this gives a proper $2$-coloring.
So $\chi_5(4)\le 2$. Since $\chi_5(4)\ge 2$, we have $\chi_5(4)=2$.
Applying Lemma~\ref{lem:recursion} with $d=4$ and iterating gives the second result.

For $q=3$ and $q=4$, a non-trivial blocking set exists in $\PG(2,q)$ (in fact it exists for every $q>2$) \cite{Richardson1956}. 
Coloring the points of this blocking set with one color and the remaining points with another color gives a proper $2$-coloring, showing that $\chi_q(3)\le 2$ for $q=3$ and $q=4$.
Since $\chi_q(2)=2$ and $\chi_q(3)\ge \chi_q(2)$, we have $\chi_q(3)=2$ for $q=3$ and $q=4$.
Applying Lemma~\ref{lem:recursion} with $d=3$ and iterating gives the third result.
\end{proof}

\begin{remark}
Our computations \cite{Bishnoi2025} produced $3$-colorings of $\PG(4,3)$ and $\PG(4,4)$.
Thus $\chi_3(5)\le 3$ and $\chi_4(5)\le 3$.
Together with the non-existence of non-trivial blocking sets in $\PG(3,3)$ and $\PG(3,4)$
\cite{Tallini1988,Cassetta1998}, this gives $
\chi_3(4)=\chi_3(5)=\chi_4(4)=\chi_4(5)=3$. 
Applying Lemma~\ref{lem:recursion} with $d=5$ also gives $\chi_q(n)\le n/2+O(1)$ for $q=3$ and $q=4$, so these computed base cases do not improve the bounds above.
\end{remark}

\section{Connection to classical Ramsey theory}\label{sec:lower}
Here we prove the lower bound on $\chi_2(n)$ given in \Cref{prop:chin-lower}.
The same proof works for any finite abelian group $G$ after fixing a strict linear order on $G$: for each unordered pair $\{u,v\}$ with $u<v$, color the edge by the color of $v-u$.
If the vertex-coloring of $G$ has no monochromatic solution to $x+y=z$ (i.e., each color class is sum-free), then the edge-coloring of the complete graph $K_{|G|}$ has no monochromatic triangles. This argument goes back to Abbott and Hanson~\cite{AbbottHanson1972} who generalized Schur's idea~\cite{Schur1917}. For completeness, we make the connection with $\chi_2(n)$ precise and provide a detailed proof. All logarithms in the statement and the proof have base two.

\begin{repproposition}{prop:chin-lower}
 $\chi_2(n)\ge(1 - o(1))\bigl(n/\log n\bigr)$.   
\end{repproposition}

\begin{proof}
We work in $\F_2^n$. 
Assume $\chi_2(n)\le k$. Then there is a coloring $c:\F_2^n\setminus\{0\}\to[k]$ such that no triple $\{x,y,x+y\}$ is monochromatic.
(Over $\F_2$ these three elements are distinct, since $x=y$ would give $x+y=0$.)

Use $c$ to color the edges of $K_{2^n}$ on vertex set $\F_2^n$ as follows:
for every edge $\{u,v\}$, color by $c(v-u)$.
This is well-defined since $v-u\neq 0$.
If $u,v,w$ formed a monochromatic triangle, then
\[
c(v-u)=c(w-v)=c(w-u),
\]
but $(v-u)+(w-v)=w-u$ in $\F_2^n$, which contradicts the assumption that $c$ has no monochromatic triple. 
Hence, this edge-coloring of $K_{2^n}$ has no monochromatic triangle, so
\begin{equation}\label{eq:Ramsey-lb}
R(3;k)\ >\ 2^n.
\end{equation}

We take the contrapositive.
If $R(3;k)\le 2^n$, then $\chi_2(n)>k$.
Fix $\alpha\in(0,1)$ and set $k=\big\lfloor \alpha\,n/\log n\big\rfloor$.
Recall that the classical upper bound of Schur and all subsequent improvements are of the form $R(3;k)=O(k!)$. Therefore
\[
R(3;k) = O(k!) =\ O(k^k) =\ O(2^{k\log k}) =\ O\left(2^{ \tfrac{\alpha\,n}{\log n} \log(\alpha\,n)}\right) =\ O(2^{\alpha n}) .
\]
For all sufficiently large $n$, it follows that $R(3;k) \le 2^n$ and hence $\chi_2(n)>k = \lfloor \alpha\,n/\log n \rfloor$. 
Since $\alpha\in(0,1)$ was arbitrary, $\chi_2(n)\ge(1 - o(1))(n/\log n)$.
\end{proof}

\subsection{Small values}\label{subsec:small-n}
We will use the known values and bounds~\cite{RadziszowskiSurvey2017} for the multicolor Ramsey numbers for triangles to obtain lower bounds on $\chi_2(n)$ for small $n$.
\[
R(3;1)=3,\qquad R(3;2)=6,\qquad R(3;3)=17,
\]
\[
51 \le R(3;4) \le 62,\qquad
162 \le R(3;5) \le 307 \ ,\qquad
538 \le R(3;6) \le 1838.
\]
These bounds give the following values of $\chi_2(n)$:
\begin{itemize}
  \item $n=5$: Since $R(3;3)=17\le 2^5$, \cref{eq:Ramsey-lb} gives $\chi_2(5)>3$. By \Cref{thm:main}, $\chi_2(5)\le 4$. Hence $\chi_2(5)=4$.
  \item $n=6,7$: Since $R(3;4)\leq 62< 2^6$, we get $\chi_2(7)\geq \chi_2(6)>4$. By \Cref{thm:main}, $\chi_2(6) = \chi_2(7) = 5$.
\end{itemize}

\begin{corollary}
$\chi_2(7)=5$.
\end{corollary}

Thus $n=7$ is the first case $n\ge 4$ where $\chi_2(n)\ne n-1$.
These small values highlight the usefulness of the Ramsey connection and also show that we recover known results for $n=5$ \cite{Rosa1970-2} and $n=6$ \cite{Fugere1994}.

We record the current small cases and the best bounds that we have obtained.
\[
\begin{array}{c|cccccccccccc}
 n & 2 & 3 & 4 & 5 & 6 & 7 & 8 & 9 & 10 & 11 & 12 & 13 \\ \hline
 \chi_2(n) & 2 & 3 & 3 & 4 & 5 & 5 & [\,5,6\,] & [\,6,7\,] & [\,6,7\,] & [\,7,8\,] & [\,7,9\,] & [\,7,9\,]
\end{array}
\]

\section{Connection to vector space Ramsey numbers}\label{sec:vsrn}

We now recall the multicolor vector space Ramsey numbers.
Let $\mathcal{L}_1(V)$ denote the set of $1$-dimensional subspaces of a finite-dimensional vector space $V$ over $\F_q$.
For integers $t_1,\dots,t_k\ge 1$, the number $R_q(t_1,\dots,t_k)$ is defined to be the smallest $n$ such that for every coloring $f:\mathcal{L}_1(\F_q^n)\to[k]$
there exist an index $i\in[k]$ and a $t_i$-dimensional subspace $U\le \F_q^n$ such that all $1$-dimensional subspaces of $U$ receive color $i$.

In the case $t_1=\dots=t_k=t$ we write $R_q(t;k):=R_q(t,\dots,t)$.
Thus, the existence of a monochromatic $(t-1)$-space in $\PG(n-1,q)$ is the same as the existence of a $t$-dimensional subspace $U \le \F_q^n$ whose $1$-dimensional subspaces are all the same color.
This gives an equivalence between $\chi_q(t;n)$ and the vector space Ramsey numbers $R_q(t;k)$.
Let $q$ be a prime power, $t\ge 2$, $n\ge t$, and $k\ge 1$.
Then
\begin{equation}\label{eq:chi-vs-Rq}
    \chi_q(t;n)\le k \quad\Longleftrightarrow\quad  R_q(t;k)>n.
\end{equation}
Equivalently,
\[
R_q(t; k) = \min \{n : \chi_q(t; n) > k\}.
\]

Frederickson and Yepremyan~\cite{Frederickson2023} note that standard applications of the Lov\'asz Local Lemma give lower bounds for $R_2(t;k)$ of order $\Omega(2^t \log k / t)$.
The same argument over $\F_q$ shows that, for all $q$ and for fixed $t$,
\[
R_q(t;k)>C_t\frac{q^{t-1}}{\log q}\log k.
\]

Combining this probabilistic lower bound with the recursion in Lemma~\ref{lem:qtrecursion2} gives an improved lower bound which is linear in $k$. 

\begin{corollary}\label{cor:LLL-boost}
For every prime power $q$, every integer $t\ge2$ and for sufficiently large $k$,
\[
R_q(t;k)>C_t\frac{q^{t-1}}{\log q}\,k
\]
for some constant $C_t>0$.
\end{corollary}
\begin{proof}
Fix an integer $k_0\ge2$. 
The Lov\'asz Local Lemma bound gives $R_q(t;k_0)>c_t q^{t-1}/\log q$ for some constant $c_t>0$.
Equivalently, there exists an integer $N\ge c_t q^{t-1}/\log q$ such that $\chi_q(t;N)\le k_0$.
By iterating Lemma~\ref{lem:qtrecursion2}, we get $\chi_q(t;jN)\le jk_0$ for every $j\ge1$.
Taking $j=\lfloor k/k_0\rfloor$ and using \eqref{eq:chi-vs-Rq}, we obtain $
R_q(t;k)>\lfloor k/k_0\rfloor\cdot N$.
\end{proof}

For $t\ge3$, we also get the following lower bounds for $R_q(t;k)$.

\begin{reptheorem}{thm:Rqt-lower}
For every prime power $q$, every integer $t\ge3$, and every integer $k\ge1$,
\[
R_q(t;k)>kt-O_t(1).
\]
\end{reptheorem}
\begin{proof}
    Use Theorem~\ref{thm:chiq-upper-bound_t} and equivalence \eqref{eq:chi-vs-Rq}.
\end{proof}
\begin{remark}
The bounds in \Cref{cor:LLL-boost} and \Cref{thm:Rqt-lower} are useful in different regimes.
For fixed $q$ and $t$, both are linear in $k$. 
When $q$ is large, \Cref{cor:LLL-boost} is stronger.
\end{remark}

For $t=2$, we provide further improvements to the lower bound.

\begin{lemma}\label{lem:Rq2k-lower}
For every integer $k\ge 1$, the following holds.
\begin{itemize}
    \item If $q=2$,
        \[
        R_2(2;k)>\frac32 k-O(1).
        \]
    \item If $q=3$ or $q=4$,
        \[
        R_q(2;k)>2k-O_q(1).
        \]
    \item If $q=5$,
        \[
        R_5(2;k)>3k-O(1).
        \]
    \item For every fixed prime power $q\ge7$,
        \[
        R_q(2;k)>\frac q2 k-O_q(1).
        \]
\end{itemize}
\end{lemma}

\begin{proof}
Assume that there exist constants $\alpha>0$ and $C$ such that
\[
  \chi_q(n)\le \alpha n+C
  \quad\text{for all }n\ge2.
\]
Fix $k>C$.
If
\[
n\le \left\lfloor\frac{k-C}{\alpha}\right\rfloor,
\]
then $\alpha n+C\le k$, so $\chi_q(n)\le k$.
By \eqref{eq:chi-vs-Rq} with $t=2$, this implies $R_q(2;k)>n$.
Hence
\[
R_q(2;k)>\frac{k}{\alpha}-O(1).
\]

For $q=2$, we use $\alpha=2/3$, which follows from \Cref{thm:main}. 
For $q=3$ and $q=4$, we use $\alpha=1/2$ from \Cref{thm:chiq-upper-bound}. 
For $q=5$, we use $\alpha=1/3$ from \Cref{thm:chiq-upper-bound}. 
For fixed $q\ge7$, we use $\alpha=2/q$ from \Cref{thm:chiq-upper-bound}. 
This gives the stated bounds.
\end{proof}

The best upper bounds on $R_q(t; k)$ in general are tower functions of height $(k - 1)(t - 1) + 1$ \cite[Theorem 1.3]{Frederickson2023}.
For $q = 2$, Hunter and Pohoata \cite{Hunter2025} used the breakthrough result of Kelley and Meka \cite{KelleyMeka2023} to show that $R_2(2; k) = O(k \log^7 k)$.
From the connection to $R(3; k)$ we immediately get the following.
\begin{reptheorem}{thm:R22k-upper}
$R_2(2;k)\le k\log k + O(1)$.
\end{reptheorem}

\begin{proof}
Let $n = R_2(2; k) - 1$. 
We assume $k \geq 2$.
Because there is a coloring of the points of $\PG(n - 1, 2)$ with $k$ colors and no monochromatic lines, \eqref{eq:Ramsey_connection} implies that 
\[
2^n < R(3; k) \leq e k! + 1 \leq 2^{k \log k + \log e}.
\]
Since $n = R_2(2; k) - 1$, this gives
\[
R_2(2; k) < \log R(3;k) + 1 < k \log k + \log e + 1.
\]

\end{proof}

\section{Further directions}\label{sec:conclusion}
\subsection{Binary projective spaces}\label{sec:open-binary}
The central problem is to determine the asymptotic behavior of $\chi_2(n)$.
\[
\text{Is }\ \chi_2(n)=\Theta(n)\ \text{ or }\ \chi_2(n)=o(n)\ ?
\]
Currently, we have
\[
\chi_2(n)\ =\ \Omega \Bigl(\frac{n}{\log n}\Bigr)
\qquad\text{and}\qquad
\chi_2(n)\ \le\ \lfloor 2n/3 \rfloor + 1.
\]
Proving $\chi_2(n)=o(n)$ would be a major breakthrough. 
A sublinear growth of $\chi_2(n)$ would force much stronger lower bounds for multicolor triangle Ramsey numbers than is currently known.
As a small step towards that goal, our recursion in~\Cref{lem:recursion} shows that a good bound $\chi_2(s)\le t$ for small values of $s$ and $t$ implies stronger asymptotic bounds, as follows.

\begin{proposition}\label{prop:main_generalized}
 For all integers $n\geq d\geq 2$,
    $$\chi_2(n)\leq \left\lfloor \frac{\chi_2(d)-1}{d-1} \ n \right \rfloor +C_d, $$
    where $1\leq C_d \leq \chi_2(d)-1$ is a constant depending only on $d$. Moreover, one can take $C_4=1$.
\end{proposition}

\begin{corollary}
For every integer $d\geq 2$,
$$\lim_{k \rightarrow \infty} R(3;k)^{1/k} \geq 2^{(d-1)/(\chi_2(d)-1)}.$$
\end{corollary}
\begin{proof}
Combine~\Cref{prop:main_generalized} with the implication in (\ref{eq:Ramsey_connection}).
\end{proof}
Applying this to $d = 4$ from Theorem~\ref{thm:main}, we obtain:
\[
\lim_{k\to\infty} R(3;k)^{1/k}\ \ge\ 2^{3/2}\ \approx\ 2.828.
\]
The best-known lower bound is
\[
\lim_{k\to\infty} R(3;k)^{1/k}\ \geq (380)^{1/5} \approx 3.2806 \; \cite{Ageron2021}
\]
As a first concrete step, it is natural to improve the constant $2/3$ in the upper bound of $\chi_2(n)$.
\[
\chi_2(8)= 5  \ \Rightarrow\ \lim_{k\to\infty} R(3;k)^{1/k}\ \ge\ 2^{7/4}\ \approx\ 3.364,
\]
\[
\chi_2(13)\le 8 \ \Rightarrow\ \lim_{k\to\infty} R(3;k)^{1/k}\ \ge\ 2^{12/7}\ \approx\ 3.2813.
\]
Thus, proving $\chi_2(8)=5$ or $\chi_2(13)=8$ would improve the current best lower bound on the growth rate of $R(3;k)$.
\begin{problem}
    Determine whether $\chi_2(8)=5$ or $\chi_2(13)=8$.
\end{problem}

\subsection{Non-binary projective spaces}\label{sec:open-nonbinary}
\begin{problem}
Fix a prime power $q\ge 3$.
Determine the asymptotic growth of $\chi_q(n)$.
Is $\chi_q(n)=\Theta(n)$?
\end{problem}

As a first step in this direction, one can try to improve the constants in our linear upper bounds.
It will be interesting to improve the constant $\alpha_q$ such that
\[
\chi_q(n)\ \le\ \alpha_q n + O(1).
\]
By \Cref{thm:chiq-upper-bound}, we have $\alpha_q\le 1/2$ for $q=3,4$, $\alpha_5\le 1/3$, and $\alpha_q\le 2/q$ for $q\ge7$.

\subsection{General setting of vector space Ramsey numbers}
As we have shown, improving the bounds in the more general setting of $\chi_q(t; n)$, will have direct consequences on the vector space Ramsey numbers $R_q(t; k)$. 

\begin{problem}
Fix a prime power $q$ and an integer $t\ge 3$.
Determine the asymptotic growth of $\chi_q(t;n)$.
Is it true that $\chi_q(t;n)=\Theta_{t,q}(n)$?
\end{problem}

If one improves the lower bounds on $\chi_q(t;n)$, then one also improves the upper bounds on vector space Ramsey numbers $R_q(t;k)$.
The current upper bounds on $R_q(t;k)$ for $q=3$ have tower growth (see \cite{Frederickson2023}), in contrast to the linear lower bounds obtained here.   
Determining the exact value of $\chi_q(t; n)$ for small values of $q, t,$ and $n$, will also be interesting, as it can improve the general upper bounds via the recursion that we have established in our work. 

\section*{Acknowledgements}
The authors thank Ferdinand Ihringer, Sam Mattheus, and Liana Yepremyan for valuable insights and discussions.

\end{document}